\documentclass[10pt,a4paper]{article}
\usepackage{amsmath}
\usepackage{latexsym,amssymb,amsmath,amsthm,amsfonts}
\usepackage{graphicx}
\usepackage{enumerate}

\newtheorem{lemma}{Lemma}[section]
\newtheorem{proposition}[lemma]{Proposition}
\newtheorem{theorem}[lemma]{Theorem}
\newtheorem{corollary}[lemma]{Corollary}

\numberwithin{equation}{section}

\newcommand{\ud}{\mathrm{d}}
\newcommand{\RR}{\mathbb{R}}
\newcommand{\f}{\frac}

\newcommand{\xx}{|x|^2}
\newcommand{\yy}{|y|^2}
\newcommand{\xy}{\langle x,y\rangle}

\newcommand{\pppp}[4]%
  {\frac{\partial^3{#1}}{\partial{#2}\partial{#3}\partial{#4}}}

\newcommand{\p}{\phi}
\newcommand{\ps}{\phi(s)}
\newcommand{\pab}{\alpha\phi\left(\frac{\beta}{\alpha}\right)}

\newcommand{\sq}{\frac{(\alpha+\beta)^2}{\alpha}}

\renewcommand{\a}{\alpha}
\renewcommand{\b}{\beta}
\newcommand{\ab}{(\alpha,\beta)}
\newcommand{\ta}{\tilde\alpha}
\newcommand{\tb}{\tilde\beta}
\newcommand{\ha}{\hat\alpha}
\newcommand{\hb}{\hat\beta}
\newcommand{\ba}{\bar\alpha}
\newcommand{\bb}{\bar\beta}

\newcommand{\aij}{a_{ij}}

\newcommand{\bi}{b_i}
\newcommand{\bj}{b_j}

\newcommand{\bij}{b_{i|j}}

\newcommand{\baij}{\bar a_{ij}}

\newcommand{\bbij}{\bar b_{i|j}}

\newcommand{\rij}{r_{ij}}
\newcommand{\sij}{s_{ij}}
\newcommand{\ri}{r_i}
\newcommand{\si}{s_i}
\newcommand{\rj}{r_j}
\newcommand{\sj}{s_j}

\newcommand{\brij}{\bar r_{ij}}

\newcommand{\bRic}{{}^{\bar{\alpha}}{\mathrm{Ric}}}

\begin{document}
\title{Douglas metrics of $\ab$ type}

\author{Changtao Yu}
\date{September 10, 2016}
\maketitle

\begin{abstract}
In this paper, the Douglas curvature of $\ab$-metrics, a special class of Finsler metrics defined by a Riemannian metric $\a$ and a $1$-form $\b$, is studied. These metrics with vanishing Douglas curvature in dimension $n\geq3$ are classified by using a new class of metrical deformations called $\b$-deformations. The result shows that conformal $1$-forms of Riemannian metrics play a key role, and an effective way to construct such $1$-forms is  provided also by $\b$-deformations.
\end{abstract}

\section{Introduction}
In 1927, J. Douglas introduced the Douglas curvature for Finsler metrics\cite{D}. Douglas curvature is an important projectively invariant in Finsler geometry. It it also a non-Riemannian quantity, since all the Riemannian metrics have vanishing Douglas curvature inherently. Finsler metrics with vanishing Douglas curvature are called {\emph Douglas metrics}. Roughly speaking, a Douglas metric is a Finsler metric which is locally projectively equivalent to a Riemannian metric\cite{Li}.% The study on Douglas metrics will enhance our understanding on the geometric meaning of non-Riemannian quantities in Finsler geometry.

Douglas metrics form a rich class of Finsler metrics including locally projectively flat Finsler metrics and Berwald metrics, the later are those metrics whose Berwald curvature vanishes\cite{BM}. It is known that a Finsler metric is locally projectively flat, namely its geodesics are all straight line segments in some suitable locally coordinate system, if and only if its both Douglas curvature and Weyl curvature vanish.

In this paper, we will focus on an important class of Finsler metrics called $\ab$-{\emph metrics}\cite{bacs-cxy-szm-curv}, which are given as
$$F=\a\phi\left(\f{\b}{\a}\right),$$
where $\a=\sqrt{a_{ij}(x)y^iy^j}$ is a Riemannian metric, $\b=b_i(x)y^i$ is a $1$-form and $\phi(s)$ is a smooth function. In 2011, the author pointed out that $\ab$-norms are those Minkowski norms which are preserved under the action of orthogonal group $O(n-1)$, which means that the indicatrix of a $\ab$-norm is a rotational hypersurface with the rotation axis passing through the origin\cite{YCT}. It is obvious that Euclidian norm is just the Minkowski norm preserved under the action of $O(n)$. In this sense, the class of $\ab$-metrics is the metrical category which is nearest to Riemannian metrics in all the Finsler metrics from the geometric point of view. Actually,$\ab$-metrics can be regards as the resulting of a Riemannian metric $\a$ disturbed by using a $1$-form $\b$, and the function $\phi$ represents the disturbing way.

Because of its excellent symmetry and computability, many nice results about $\ab$-metrics have been acquired in the past forty years\cite{bacs-cxy-szm-curv}. Moreover, $\ab$-metrics provides several wonderful metrical models for physics and biology\cite{aim-ttos,cl}, and the most important one is the so-called Randers metrics\cite{Ran}.

Randers metrics, given as $F=\a+\b$~(the corresponding function is $\phi(s)=1+s$), are the solutions of Zermelo's navigation problem\cite{db-robl-szm-zerm}. In the original discussions of the physicist G.~Randers, $\a$ and $\b$  represent a gravitational field and an electromagnetic field respectively. It is known that a Randers metrics $F=\a+\b$ is a Douglas metric if and only if $\b$ is closed\cite{BM}.

\begin{figure}[h]
  \centering
  \includegraphics{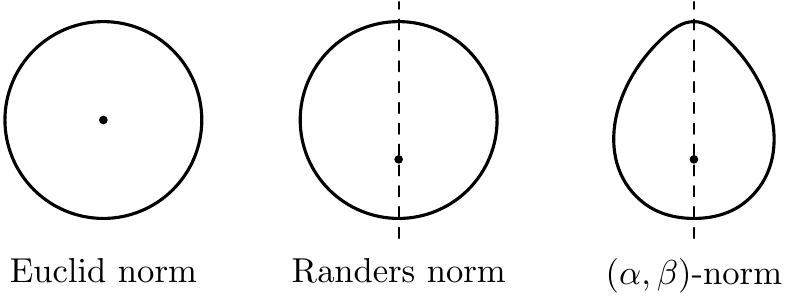}
\end{figure}

The aim of this paper is to provide a luminous characterization for Douglas $\ab$ metrics based on a result given by B. Li, Z. Shen and Y. Shen. In 2009, they proved a characterization\cite{Li}:
Let $F=\pab$ be a Finsler metric on on open subset $U\in\RR^n$ with dimension $n\geq3$. Assume that (a) $F$ is not of Randers type, i.e., $F\neq\sqrt{c_1\a^2+c_2\b^2}+c_3\b$ for any constants $c_1$, $c_2$ and $c_3$, (b) $\b$ is not parallel with respect to $\a$, (c) $\ud b\equiv0$ or $\ud b\neq0$ everywhere, then $F$ is a Douglas metric if and only if $\ps$ satisfies the following ODE,
\begin{eqnarray}\label{odeDouglasab}
\left\{1+(k_1+k_3)s^2+k_2s^4\right\}\p''(s)=(k_1+k_2s^2)\{\ps-s\p'(s)\},
\end{eqnarray}
where $k_1,k_2,k_3$ are constants with $k_2\neq k_1k_3$, and at the same time $\b$ satisfies
\begin{eqnarray}
%\G&=&\xi y^i-\tau(k_1\a^2+k_2\b^2)b^i,\label{Gi}\\
\bij=\tau\left\{(1+k_1b^2)a_{ij}+(k_3+k_2b^2)b_ib_j\right\}\label{bijDouglasab}
\end{eqnarray}
for some scalar function $\tau(x)$.

As an example, one can see that the function $\phi(s)=(1+s)^2$ satisfies Equation (\ref{odeDouglasab}) with $(k_1,k_2,k_3)=(2,0,-3)$. The corresponding $\ab$-metrics $F=\sq$ are called square metrics sometimes or Berwald type of metrics, since the following famous metric
\begin{eqnarray}\label{BerwaldF}
F=\f{(\sqrt{(1-|x|^2)|y|^2+\langle x,y\rangle^2}+\langle x,y\rangle)^2}{(1-|x|^2)^2\sqrt{(1-|x|^2)|y|^2+\langle x,y\rangle^2}}
\end{eqnarray}
was constructed firstly by L.~Berwald in 1929\cite{Be}. This metric defined on the unit ball $\mathbb{B}^n(1)$ is projectively flat with vanishing flag curvature and hence is a Douglas metric naturally.

Square metrics may be the most important kind of $\ab$-metrics except for Randers metrics, because it is the rare metrics to be of great geometric properties. Recently, E. Sevim et al. studied Einstein $\ab$-metrics of Douglas type, and proved that such Finsler metrics must be either a Randers metric or a square metric\cite{CT,SSZ}.  Later on, Z. Shen and the author provided a concise characterization of Einstein square metric\cite{szm-yct-oesm}. More specifically, if $F=\sq$ is an Einstein-Finsler metric, then the Riemannian metric $\ba:=(1-b^2)\a$ and the $1$-form $\bb:=\sqrt{1-b^2}\b$ satisfy
\begin{eqnarray*}
\bRic=-(n-1)k^2\ba,\qquad\bbij=k\sqrt{1+\bar b^2}\; \baij,
\end{eqnarray*}
or equivalently, $\ba:=(1-b^2)^\frac{3}{2}\sqrt{\a^2-\b^2}$ and $\bb:=(1-b^2)^2\b$ satisfy
\begin{eqnarray*}
\bRic=0,\qquad\bbij=k\baij,
\end{eqnarray*}
where $k$ is a constant and $\bar b:=\|\bb\|_{\ba}$. In the above two cases $F$ can be reexpressed as
\begin{eqnarray}\label{expression1}
F=\f{(\sqrt{1+\bar b^2}\; \ba+\bb)^2}{\ta}
\end{eqnarray}
and
\begin{eqnarray}\label{expression2}
F=\f{(\sqrt{(1-\bar b^2)\ba^2+\bb^2}+\bb)^2}{(1-\bar b^2)^2\sqrt{(1-\bar b^2)\ba^2+\bb^2}}.
\end{eqnarray}
respectively. Both of them belong to a larger metrical category called {\emph general $\ab$-metrics}\cite{YCT}.

The key technique for discussing Einstein square metrics is a new kinds of metrical deformations called $\b$-deformations developed by the author\cite{yct-dhfp}, and it is also appropriate for our problem. $\b$-deformations, determined by a Riemannian metric $\a$ and a $1$-form $\b$, generalize the classical navigation deformation for Randers metrics. It has been applied successfully for many problems about $\ab$-metrics even including Riemannian metrics\cite{yct-odfr,yct-odfa}.

The effect of $\b$-deformations is to make clear the underlying geometry. Take square metrics for example. According to (\ref{bijDouglasab}), a square metric $F=\sq$ is a Douglas metric if and only if
\begin{eqnarray*}
%\G&=&\xi y^i-\tau(k_1\a^2+k_2\b^2)b^i,\label{Gi}\\
\bij=\tau\left\{(1+2b^2)a_{ij}-3b_ib_j\right\}.
\end{eqnarray*}
However, the geometric meaning of $\b$ is very obscure. By taking the some $\b$-deformations mentioned above, one can see that a square metric is a Douglas metric if and only if it is given in the form (\ref{expression1}) or (\ref{expression2}) in which $\bb$ is closed and conformal with respect to $\ba$ according to Theorem \ref{mainDouglasab}. That is to say, the property about $\b$ becomes clear after $\b$-deformations.

The main result in this paper is the following classification theorem for Douglas $\ab$-metrics when $n\geq3$. The case when $n=2$ has been researched in \cite{YGJ}.
\begin{theorem}[Classification]\label{douglasab}
Let $F=\pab$ be a $\ab$-metric on a $n$-dimensional manifold $M$ with $n\geq3$. Then $F$ has vanishing Douglas curvature if and only if $F$ lies in one of the following cases:
\begin{enumerate}
\item $\b$ is parallel with respect to $\a$. In this case, $F$ has vanishing Douglas curvature for any suitable functions $\phi(s)$;
\item $F$ is a Douglas-Randers metric;
\item On the open subset $\mathcal U$ of $M$ where $\ud b\neq0$ everywhere or $\ud b\equiv0$, $F$ can be reexpressed (if necessary) still as the form $F=\pab$ such that one of the following holds
\begin{enumerate}
\item $\ps$ is given by
\begin{eqnarray}\label{D1p}
\phi_\sigma=\mathrm{hypergeom}\left(\left[-\f{1}{2},-\sigma\right],\left[\f{1}{2}\right],s^2\right)+\varepsilon s,\quad\sigma\neq0,-\frac{1}{2}.
\end{eqnarray}
$\a$ and $\b$ are determined by
\begin{eqnarray*}
\a=(1-\bar
b^2)^{-\sigma-\frac{1}{2}}\ba,\qquad
\b=(1-\bar b^2)^{-\sigma-\frac{1}{2}}\bb;
\end{eqnarray*}
\item $\ps$ is given by
\begin{eqnarray}\label{D1z}
\phi_\sigma=e^{-\sigma s^2}-\sqrt{\sigma\pi}s\,\mathrm{erf}\,(\sqrt{\sigma}s)+\varepsilon s,\quad\sigma=\pm1.
\end{eqnarray}
$\a$ and $\b$ are determined by
\begin{eqnarray*}
\a=e^{\sigma\bar b^2}\ba,\qquad\b=e^{\sigma\bar b^2}\bb;
\end{eqnarray*}
\item $\ps$ is given by
\begin{eqnarray}\label{D1n}
\phi_\sigma(s)=1+\varepsilon s+\sum_{n=1}^{\infty}a_{2n}s^{2n},\quad|\sigma|<1,
\end{eqnarray}
where $\{a_{2n}\}$ is determined by the recursion formula
$$a_{2n+2}=-2\sigma\f{2n(2n-1)}{(2n+2)(2n+1)}a_{2n}-\f{(2n-1)(2n-3)}{(2n+2)(2n+1)}a_{2n-2}$$
with the initial terms $a_0=1$ and $a_2=0$. $\a$ and $\b$ are determined by
\begin{eqnarray*}
\a=\frac{\exp\left(-\frac{\sigma}{2\sqrt{1-\sigma^2}}\arctan\frac{\sigma+\bar
b^2}{\sqrt{1-\sigma^2}}\right)}
{(1+2\sigma\bar b^2+\bar b^4)^\frac{1}{4}}\ba,\qquad
\b=\frac{\exp\left(-\frac{\sigma}{2\sqrt{1-\sigma^2}}\arctan\frac{\sigma+\bar
b^2}{\sqrt{1-\sigma^2}}\right)} {\left(1+2\sigma\bar b^2+\bar
b^4\right)^\frac{1}{4}}\bb.
\end{eqnarray*}
\end{enumerate}
In (a)-(c), $\ba$ is an arbitrary Riemannian metric, $\bb$ is a closed $1$-form which is conformal with respect to $\ba$, and $\bar b:=\|\bb\|_{\ba}$.
\end{enumerate}
\end{theorem}

The function (\ref{D1p}) with $\sigma=1$ and $\epsilon=2$ is $\phi(s)=(1+s)^2$ in fact. By Theorem \ref{douglasab} we know that a square metric $F=\sq$
can be also expressed as
\begin{eqnarray}\label{simpleep}
F=\f{(\ba+\bb)^2}{(1-\bar b^2)^\frac{3}{2}\ba},
\end{eqnarray}
and it is a Douglas metric if and only if $\bb$ is closed and conformal with respect to $\ba$. Actually, as we have seen on square metrics, there are infinity many way to express a given $\ab$-metric~(one can find more discussions in \cite{yct-dhfp}). However, there is an inherent relevance between such two different expressions. Take (\ref{expression2}) and (\ref{simpleep}) for example. According to Proposition \ref{weingiengin}, if $\b$ is closed and conformal with respect to $\a$, then $\bb=\frac{\b}{\sqrt{1-b^2}}$ is closed and conformal with respect to $\ba=\frac{\sqrt{(1-b^2)\a^2+\b^2}}{1-b^2}$~(the corresponding factors in Proposition \ref{weingiengin} are $\kappa=-\frac{1}{1-b^2}$, $\rho=0$ and $\nu=\frac{1}{\sqrt{1-b^2}}$). Notice that $\bar b=b$ in this case, plugging all of them into (\ref{simpleep}) will lead to the expression (\ref{expression2}).

In order to construct Douglas $\ab$-metrics, the most important thing is to construct Riemannian metrics and their non-trivial conformal $1$-forms. In Section \ref{1f} we provide an effective method. Take also square metric for example. Suppose that it is expressed in the form (\ref{simpleep}), then by Corollary \ref{womeingengiengi} we know that all the metrics below
\begin{eqnarray*}
F=\f{\left(e^\rho\sqrt{\yy-\kappa\xy^2}+Ce^{2\rho}\sqrt{1-\kappa\xx}\xy\right)^2}
{\left(1-C^2e^{2\rho}\xx\right)^\frac{3}{2}e^\rho\sqrt{\yy-\kappa\xy^2}}
\end{eqnarray*}
are Douglas metrics for any functions $\kappa(\xx)$, $\rho(\xx)$ and non-zero constant $C$. All of them belong to the so-called spherically symmetric Finsler metrics\cite{Mo}. In particular, if we take $\kappa=\frac{\mu}{1+\mu\xx}$, $\rho=0$ and $C=1$, then we can get a series of Douglas square metrics below:
\begin{eqnarray*}
F=\f{(\sqrt{(1+\mu|x|^2)|y|^2-\mu\langle x,y\rangle^2}+\langle x,y\rangle)^2}{\sqrt{(1+\mu\xx)(1-|x|^2)^3}\sqrt{(1+\mu|x|^2)|y|^2-\mu\langle x,y\rangle^2}}.
\end{eqnarray*}
It is obvious that the corresponding metric when $\mu=-1$ is just the Berwald's original metric (\ref{BerwaldF}). More analytical examples can be constructed by Theorem \ref{douglasab} and Corollary \ref{womeingengiengi}.

\section{Basic conceptions of Finsler geometry and $\b$-deformations}\label{basic}
A Finsler metric $F$ on a manifold $M$ is a positive homogeneous continuous function
$F:TM\to[0,+\infty)$ where $F$ is smooth  on the slit tangent bundle $TM_o$ and the Hessian matrix $g_{ij}:=[\frac{1}{2}F^2]_{y^iy^j}$ is positive definite at any point $(x,y)\in TM_o$. Here $(x^i,y^i)$ denote the natural system of coordinates of $TM$.

Any Finsler metric $F$ on $M$ induces a spray vector field $G:=y^i\frac{\partial}{\partial x^i}-2G^i\frac{\partial}{\partial y^i}$ on $TM_o$. The projection of the integral curves of $G$ from $TM_o$ to $M$ determines the geodesics of the Finsler manifold $(M,F)$. The coefficients
$$G^i(x,y):=\f{1}{4}g^{il}\left\{[F^2]_{x^ky^l}y^k-[F^2]_{x^l}\right\}$$
are called the {\emph spray coefficients} of $F$, where $(g^{ij})$ is the inverse of $(g_{ij})$. If $F$ is a Riemannian metric, then $G^i$ are determined by its Christoffel symbols as $G^i(x,y)=\frac{1}{2}\Gamma^i{}_{jk}(x)y^jy^k$.

By definition, $F$ is called a {\emph Berwald metric} if $G^i$ are quadratic in $y\in T_xM$ at every point $x$, i.e.,
$$G^i(x,y)=\frac{1}{2}\Gamma^i{}_{jk}(x)y^jy^k.$$
$F$ is called a {\emph Douglas metric} if $G^i$ are in the following form
$$G^i(x,y)=\frac{1}{2}\Gamma^i{}_{jk}(x)y^jy^k+P(x,y)y^i,$$
or equivalently, if its Douglas tensor $\mathcal D:=D_j{}^i{}_{kl}\,\ud x^j\otimes\frac{\partial}{\partial x^i}\otimes\ud x^k\otimes\ud x^l$ on $TM_o$ vanishes, where
$$D_j{}^i{}_{kl}:=\f{\partial^3}{\partial y^j\partial y^k\partial y^l}\left(G^i-\f{1}{n+1}\f{\partial G^m}{\partial y^m}y^i\right).$$
The following diagram shows clear the relationship of the related metrical categorys.
$$\{\textrm{Riemann metrics}\}\subset\{\textrm{Berwald metrics}\}\subset\{\textrm{Douglas metrics}\}.$$

On the other hands, $\b$-{\emph deformations} are a triple kinds of deformations in terms of a given Riemannian metric $\a$ and a $1$-form $\b$ as follows,
\begin{eqnarray*}
&\ta=\sqrt{\a^2-\kappa(b^2)\b^2},\qquad\tb=\b;\label{b1}\\
&\ha=e^{\rho(b^2)}\ta,\qquad\hb=\tb;\label{b2}\\
&\ba=\ha,\qquad\bb=\nu(b^2)\hb.\label{b3}
\end{eqnarray*}
Be attention that the factor $\kappa$ must satisfy an additional condition $1-\kappa b^2>0$ to keep the resulting metric positive definite.

%Obviously, the first kind of $\b$-deformation can be regarded as stretch change for $\a$ along the direction determined by $\b$, the second one is conformal change and the third one is length change for $\b$\cite{yct-dhfp}. Especially, all the deformation factors are just functions of the length of $\b$.% Notice that we choose $b^2$ instead of $b$ as the variable here, because it will be convenient for computations.

The abbreviations below are frequently-used in the literatures about $\ab$-metrics:
\begin{eqnarray*}
r_i:=b^jr_{ji},\quad r:=b^ir_i,\quad s_i:=b^js_{ji},
\end{eqnarray*}
where $\rij$ and $\sij$ are the symmetrization and antisymmetrization of $\bij$ respectively, i.e.,
$$\rij:=\f{1}{2}(\bij+b_{j|i}),\quad\sij:=\f{1}{2}(\bij-b_{j|i}).$$
It is clear that $\sij=0$ if and only if $\b$ is closed.

Finally, in this paper we just need a simple formula of $\b$-deformations below, which is a direct conclusion of Lemma 2, Lemma 3 and Lemma 4 in \cite{yct-dhfp}. Notice that $\brij=\frac{1}{2}(\bar b_{i|j}+\bar b_{j|i})$ in which $\bar b_{i|j}$ means the covariant derivative of $\bar b_i$ with respect to the corresponding Riemannian metric $\ba$.
\begin{proposition}\label{relationunderdeformations}
After $\b$-deformations,
\begin{eqnarray}
\brij&=&\f{\nu}{1-\kappa b^2}\rij+\f{\kappa\nu}{1-\kappa b^2}(\bi\sj+\bj\si)-\f{\kappa'\nu}{1-\kappa b^2}r\bi\bj+\f{2\rho'\nu}{1-\kappa b^2}r(\aij-\kappa\bi\bj)\nonumber\\
&&+\left(\f{\kappa'\nu b^2}{1-\kappa b^2}-2\rho'\nu+\nu'\right)\left\{\bi(\rj+\sj)+\bj(\ri+\si)\right\}.\label{iwneingien}
\end{eqnarray}
\end{proposition}

\section{Deformations for data $\ab$ satisfying (\ref{bijDouglasab})}
Before our discussions, two facts should be pointed out. First, if $\b$ is parallel with respect to $\a$, then $F=\a\phi(\frac{\b}{\a})$ is always a Douglas metric for any function $\phi$ and any metric $\a$. So, this is a trivial case. Second, when $k_2=k_1k_3$, the solutions of (\ref{odeDouglasab}) are $\phi(s)=\sqrt{c_1\a^2+c_2\b^2}+c_3\b$. So the condition $k_2\neq k_1k_3$ ensures the corresponding $\ab$-metric is not of Randers type.

Assume that $F=\pab$ is a non-Randers type Douglas metric and $\b$ is not parallel with respect to $\a$, then $\b$ satisfies (\ref{bijDouglasab}) for some constants $k_1$, $k_2$ and $k_3$. It is obvious that the geometric meaning of $\b$ is very obscure. At least, such property about $\b$ is non-linear. That is to say, if you have two $1$-forms of such property, their linear combination  will not be of the same property any more. Hence, we would like to make clear the underlying property of $\b$, and it can be realized just by taking some suitable $\b$-deformations.

Firstly, if $\b$ satisfies (\ref{bijDouglasab}), then $\sij=0$ and hence $\b$ is closed. As a result, $\si=0$. Moreover, we have
\begin{eqnarray*}
\ri=\tau\{1+(k_1+k_3)b^2+k_2b^4\}\bi,\qquad r=\tau\{1+(k_1+k_3)b^2+k_2b^4\}b^2.
\end{eqnarray*}
Plugging all the equalities above into (\ref{iwneingien}) yields
\begin{eqnarray*}
\brij&=&\f{\tau\nu}{1-b^2\kappa}\Big\{1+k_1b^2+2[1+(k_1+k_3)b^2+k_2b^4]b^2\rho'\Big\}(\aij-\kappa\bi\bj)\\
&&+\f{\tau}{1-b^2\kappa}\Big\{[(1+k_1b^2)\kappa+k_3+k_2b^2]\nu+[1+(k_1+k_3)b^2+k_2b^4]b^2\kappa'\nu\\
&&-4[1+(k_1+k_3)b^2+k_2b^4](1-b^2\kappa)\rho'\nu
+2\{1+(k_1+k_3)b^2+k_2b^4\}(1-b^2\kappa)\nu'\Big\}\bi\bj.
\end{eqnarray*}
Notice that the rank of $\{\aij-\kappa\bi\bj\}$ is $n$ and rank of $\{\bi\bj\}$ is $1$. Hence, $\bb$ is conformal with respect to $\ba$, which means that $\brij=\sigma(x)(\aij-\kappa\bi\bj)$ for some scalar function $\sigma$, if and only if
\begin{eqnarray*}\label{woemig}
&&2\left\{1+(k_1+k_3)b^2+k_2b^4\right\}(1-\kappa b^2)\nu'\\
&=&-\left\{(1+k_1b^2)\kappa+k_3+k_2b^2+[1+(k_1+k_3)b^2+k_2b^4]b^2\kappa'-4[1+(k_1+k_3)b^2+k_2b^4](1-b^2\kappa )\rho'\right\}\nu,
\end{eqnarray*}
or equivalently,
$$\f{\nu'}{\nu}=2\rho'-\f{\kappa+b^2\kappa'}{2(1-b^2\kappa)}-\f{k_3+k_2b^2}{2\left\{1+(k_1+k_3)b^2+k_2b^4\right\}}.$$
Solving the above equation, we can obtain the following result immediately.
\begin{proposition}\label{wiengieng}
If $\b$ satisfies the condition (\ref{bijDouglasab}), then $\bb$ is closed and conformal with respect to $\ba$ after $\b$-deformations if and only if the deformation factors satisfy
\begin{eqnarray}\label{wienigehiimeig}
\nu=C\sqrt{1-b^2\kappa}e^{2\rho}e^{-\int\f{k_3+k_2b^2}{2\{(1+(k_1+k_3)b^2+k_2b^4\}}\,\ud b^2},
\end{eqnarray}
where $C$ is a non-zero constant.
\end{proposition}

The discussions up to now show that if $F=\pab$ is a non-Randers type Douglas $\ab$-metric, then we can always take some $\b$-deformations for $\a$ and $\b$ to make sure that the resulting $1$-form $\bb$ is conformal with respect to $\ba$. We hope the reverse holds too. That is to say, we hope the required data $\ab$ satisfying (\ref{bijDouglasab}) can be obtained by the data $(\ba,\bb)$ in which $\bb$ is closed and conformal with respect to $\ba$.  It can be realized easily as long as the deformations are reversible.

It is easy to verify that $\bar a^{ij}=e^{-2\rho}(a^{ij}+\frac{\kappa}{1-b^2\kappa}b^ib^j)$, so
\begin{eqnarray}\label{bbb}
\bar b^2:=\|\bb\|^2_{\ba}=\f{\nu^2e^{-2\rho}}{1-b^2\kappa}b^2=C^2e^{2\rho}e^{-\int\f{k_3+k_2b^2}{\{(1+(k_1+k_3)b^2+k_2b^4\}}\,\ud b^2}b^2.
\end{eqnarray}
In order to ensure the reversibility of deformations, we just need to ask $\bar b^2$~(as the function of $b^2$)~has inverse function. Moreover, we hope the inverse function will be as simple as possible. So, we {\emph can} choose
\begin{eqnarray}\label{rho}
\rho=\int\f{k_3+k_2b^2}{2\{(1+(k_1+k_3)b^2+k_2b^4\}}\,\ud b^2
\end{eqnarray}
and $C=1$ here. In this case, $\bar b^2=b^2$ by (\ref{bbb}) and
\begin{eqnarray}\label{nu}
\nu=\sqrt{1-b^2\kappa}e^{\rho}
\end{eqnarray}
by (\ref{wienigehiimeig}).

Now, we can choose any deformation factor $\kappa$ satisfing $1-b^2\kappa>0$ to make sure the positive definiteness of $\ba$. Lemma 1 in \cite{yct-dhfp} shown that if $F=\pab$ is a {\emph regular} Finsler metric with $\ps$ satisfying (\ref{odeDouglasab}), then the following inequalities
\begin{eqnarray}\label{womeing}
1+k_1s^2>0,\quad 1+(k_1+k_3)s^2+k_2s^4>0,\quad\forall|s|\leq b<b_o
\end{eqnarray}
hold. According to the second inequality above, we can take
\begin{eqnarray}\label{kappa1}
\kappa=-(k_1+k_3+k_2b^2).
\end{eqnarray}

Using the deformation factors $(\kappa,\rho,\nu)$ determined by (\ref{kappa1}), (\ref{rho}) and (\ref{nu}), we have the following characterization for non-Randers type Douglas $\ab$-metrics.
\begin{theorem}\label{mainDouglasab}
Let $F=\pab$ be a Finsler metric on a $n$-dimensional manifold $M$ with $n\geq3$. Suppose that the function $\ps$ satisfies (\ref{odeDouglasab}). Then $F$ is a Douglas metric if and only if $\a$ and $\b$ can be expressed as
\begin{eqnarray*}
\a=\eta(\bar b^2)\sqrt{\ba^2-\f{(k_1+k_3+k_2\bar
b^2)}{1+(k_1+k_3)\bar
b^2+k_2\bar b^4}\bb^2},\qquad
\b=\f{\eta(\bar b^2)}{\sqrt{1+(k_1+k_3)\bar b^2+k_2\bar
b^4}}\bb,
\end{eqnarray*}
where $\ba$ is an arbitrary Riemannian metric, $\bb$ is a closed $1$-form and conformal with respect to $\ba$, $\bar b:=\|\bb\|_{\ba}$.
The function $\eta(\bar b^2)$ is determined by the coefficients $k_1,k_2,k_3$ and given in the following five cases,
\begin{enumerate}[(1)]
\item When $k_2=0,~k_1+k_3=0$,
$$\eta(\bar b^2)=\exp\left\{-\f{k_3\bar b^2}{2}\right\};$$
\item When $k_2=0,~k_1+k_3\neq0$,
$$\eta(\bar b^2)=\left\{1+(k_1+k_3)\bar b^2\right\}^{-\f{k_3}{2(k_1+k_3)}};$$
\item When $k_2\neq0,~\Delta_1>0$,
$$\eta(\bar b^2)=\f{\left\{\f{\sqrt{\Delta_1}+k_1+k_3}{\sqrt{\Delta_1-k_1-k_3}}\cdot\f{\sqrt\Delta_1-k_1-k_3-2k_2\bar b^2}{\sqrt\Delta_1+k_1+k_3+2k_2\bar
b^2}\right\}^\f{k_1-k_3}{4\sqrt\Delta_1}}{\sqrt[4]{1+(k_1+k_3)\bar
b^2+k_2\bar b^4}};$$
\item When $k_2\neq0,~\Delta_1=0$,
$$\eta(\bar b^2)=\f{\sqrt{2}\exp\left\{\f{k_3-k_1}{k_1+k_3}\left[\f{1}{2+(k_1+k_3)\bar b^2}-\f{1}{2}\right]\right\}}{\sqrt{2+(k_1+k_3)\bar b^2}};$$
\item When $k_2\neq0,~\Delta_1<0$,
$$\eta(\bar b^2)=\f{\exp\left\{\f{k_1-k_3}{2\sqrt{-\Delta_1}}\left(\arctan\f{k_1+k_3+2k_2\bar
b^2}{\sqrt{-\Delta_1}}-\arctan\f{k_1+k_3}{\sqrt{-\Delta_1}}\right)\right\}}{\sqrt[4]{1+(k_1+k_3)\bar b^2+k_2\bar b^4}},$$
\end{enumerate}
where $\Delta_1:=(k_1+k_3)^2-4k_2$.
\end{theorem}
\begin{proof}
Since $\bar b^2=b^2$, $\bb=\nu\b$ combining with (\ref{nu}) yields
$$\b=\nu^{-1}\bb=\frac{\eta(\bar b^2)}{\sqrt{1-\kappa(\bar b^2)\bar b^2}}\bb=\f{\eta(\bar b^2)}{\sqrt{1+(k_1+k_3)\bar b^2+k_2\bar
b^4}}\bb,$$
in which
$$\eta(\bar b^2):=e^{-\rho(\bar b^2)}=e^{-\int\f{k_3+k_2\bar b^2}{2\{(1+(k_1+k_3)\bar b^2+k_2\bar b^4\}}\,\ud\bar b^2}$$
can be determined by elementary integrations directly. Finally, by $\ba=e^{\rho}\sqrt{\a^2-\kappa\b^2}$ we have
$$\a=\sqrt{e^{-2\rho}\ba^2+\kappa\b^2}=\eta(\bar b^2)\sqrt{\ba^2+\frac{\kappa(\bar b^2)}{1-\bar b^2\kappa(\bar b^2)}\bb^2}=\eta(\bar b^2)\sqrt{\ba^2-\f{(k_1+k_3+k_2\bar
b^2)}{1+(k_1+k_3)\bar
b^2+k_2\bar b^4}\bb^2}.$$
\end{proof}

Actually, the triple data $(\kappa,\rho,\nu)$ used above is just what was used in \cite{yct-dhfp}. In that paper, the author proved such conclusion: if $F=\pab$ is a non-trivial and non-Randers type $n$-dimensional~$(n\geq3)$ locally projectively flat Finsler metric~(which implies $\phi$ and $\b$ should satisfy the same conditions (\ref{odeDouglasab}) and (\ref{bijDouglasab}), just as Douglas $\ab$-metrics), then after the specific $\b$-deformations due to these factors, the resulting Riemannian metric $\ba$ is locally projectively flat, and $\bb$ is closed and conformal with respect to $\ba$ (See Theorem 1.2 in \cite{yct-dhfp} for details). That is to say, the only distinction between local projectively flat $\ab$-metrics and Douglas $\ab$-metrics is that we don't need any constraint condition on $\ba$ for the latter. Such phenomenon is natural because Douglas metrics include all the locally projectively flat metrics.

However, Proposition \ref{wiengieng} make it possible for us to provide some other characterizations for Douglas $\ab$-metrics. Obviously, we can take $\kappa=0$, and in this case $\nu$ is equal to $e^\rho$ due to (\ref{nu}). Here we still take $\rho$ as (\ref{rho}), just because such deformation factor make the relationship between $\bar b^2$ and $b^2$ simplest. As a result, this triple data $(\kappa,\rho,\nu)$ will lead to a different equivalent characterization as follows.
\begin{theorem}\label{main3}
Let $F=\pab$ be a Finsler metric on a $n$-dimensional manifold $M$ with $n\geq3$. Suppose that the function $\ps$ satisfies (\ref{odeDouglasab}). Then $F$ is a Douglas metric if and only if $\a$ and $\b$ can be expressed as
\begin{eqnarray*}
\a=\eta(\bar b^2)\ba,\qquad
\b=\eta(\bar b^2)\bb,
\end{eqnarray*}
where $\ba$ is an arbitrary Riemannian metric, $\bb$ is a closed $1$-form and conformal with respect to $\ba$, $\bar b:=\|\bb\|_{\ba}$.
The function $\eta(\bar b^2)$ is determined by the same way in Theorem \ref{mainDouglasab}.
\end{theorem}

It is obvious that the characterization above is simpler than Theorem \ref{mainDouglasab}, although such description for Douglas $\ab$-metrics is far from the property of projective flatness. If necessary, one can choose other triple of deformation factors satisfying (\ref{wienigehiimeig}) and obtain other characterizations. For instance, one can take $\kappa=-k_1$ according to the first inequality in (\ref{womeing}).

\section{Proof of Theorem \ref{douglasab}}
The expression for any given $\ab$-metric is non-uniqueness. In order to see that, let's introduce two transformations for the function $\phi$:
$$g_u(\ps):=\sqrt{1+us^2}\phi(\f{s}{1+us^2}),\quad h_v(\ps):=\p(vs),$$
where $u$ and $v$ are constants with $v\neq0$. It is easy to verify that
\begin{eqnarray*}
g_{u_1}\circ g_{u_2}=g_{u_1+u_2},\quad h_{v_1}\circ h_{v_2}=h_{v_1v_2},\quad h_v\circ g_u=g_{uv^2}\circ h_v.
\end{eqnarray*}
With the above generation relationships, the transformations $g_u$ and $h_v$ generate a transformation group $G$ acting on the set of all the suitable functions $\phi$ for $\ab$-metrics $F=\pab$. Two functions $\phi_1(s)$ and $\phi_2(s)$ is equivalent under the action of $G$ if and only if they provide the same type of $\ab$-metrics. For instance, all the functions equivalent to $1+s$ will provide Randers type metrics. It is obvious that such functions are given by $\ps=\sqrt{1+us^2}+vs=g_u\circ h_v(1+s)$. Notice that the non-negativity for Finsler metrics asks $\phi(s)$ to be a positive function. Hence, we can always assume $\phi(0)=1$ after necessary scaling.

If $\phi$ is the solution of Equation (\ref{odeDouglasab}), then all the functions which are equivalent to $\phi$ will still satisfy (\ref{odeDouglasab}) due to Li-Shen-Shen's result. That is to say,  the set
$$\Phi:=\{\ps~|~\ps~\textrm{satisfies Eqn. (\ref{odeDouglasab}) and}~\p(0)=1\}$$
is closed under the action of $G$. More specifically, if $\psi(s)=(g_u\p)(s)$ where $\ps\in\Phi$ , then $\psi(s)$ is the solution of
$$\left\{1+(k_1'+k_3')s^2+k_2's^4\right\}\psi''(s)=(k_1'+k_2's^2)\left\{\psi(s)-s\psi'(s)\right\}$$
with the initial conditions $\psi(0)=1$ and $\psi'(0)=\epsilon$, where the constants $k_1'$, $k_2'$ and $k_3'$ are given by
$$k_1'=k_1+u,~k_3'=k_3+u,~k_2'=k_2+(k_1+k_3)u+u^2.$$
And iff $\varphi(s)=(h_v\p)(s)$, then $\varphi(s)$ is the solution of
$$\left\{1+(k_1''+k_3'')s^2+k_2''s^4\right\}\varphi''(s)=(k_1''+k_2''s^2)\left\{\varphi(s)-s\varphi'(s)\right\}$$
with the initial conditions $\varphi(0)=1$ and $\varphi'(0)=v\epsilon$, where the constants $k_1''$, $k_2''$ and $k_3''$ are given by
$$k_1''=v^2k_1,~k_3''=v^2k_3,~k_2''=v^4k_2.$$

Define three variables depended on $\ps$:
$$\Delta_1:=(k_1+k_3)^2-4k_2,\quad\Delta_2:=4(k_1k_3-k_2),\quad\Delta_3:=k_1-k_3.$$
It is obviously that $\Delta_1-\Delta_2=\Delta_3^2$, and $\Delta_2=0$ if and only if $F=\pab$ is of Randers type.

When $\Delta_2\neq0$, the author had prove In \cite{yct-dhfp} that the set $\Phi$ is one-to-one correspondence to the quadruple data $(k_1,k_2,k_3,\epsilon)$, and a couple of variables $(p,q)_\p$ determined below
$$(p,q)_\p:=\left(\f{\sqrt{\Delta_2}}{\Delta_3},\f{\epsilon^4}{\Delta_2}\right)$$
are complete system of invariants under the action of the transformations group $G$. Some special case of $(p,q)_\p$ are defined below:
\begin{itemize}
\item $p:=0$ when $\Delta_2=0$;
\vspace{-4bp}
\item $p:=\infty$ when $\Delta_2>0$ and $\Delta_3=0$;
\vspace{-4bp}
\item $p:=i\infty$ when $\Delta_2<0$ and $\Delta_3=0$;
\vspace{-4bp}
\item $q:=0$ when $\epsilon=0$;
\vspace{-4bp}
\item $q:=\infty$ when $\epsilon\neq0$ and $\Delta_2=0$.
\end{itemize}
It is easy to see that $(0,0)_\phi$ and $(0,\infty)_\phi$ correspond to Riemannian metrics and non-Riemann Randers metrics respectively.

Moreover, under the action of $G$, the three-parameter equation (\ref{odeDouglasab}) can be simplified as an one-parameter equation, with the reduction depending on the sign of $\Delta_1$. The process is complete provided in \cite{yct-dhfp}, so we just list the essentials here.

\noindent\textbf{Case 1}: $\Delta_1>0$

In this case, (\ref{odeDouglasab}) can be reduced as
\begin{eqnarray}\label{D1pa}
\left\{1-s^2\right\}\p''(s)=2\sigma\{\ps-s\p'(s)\},\quad\sigma\neq0,-\frac{1}{2},
\end{eqnarray}
where the restriction on $\sigma$ ensures $\Delta_2\neq0$.

The solutions of (\ref{D1pa}) are given by
\begin{eqnarray*}
\p_\sigma=1+\epsilon s+\sum_{n=1}^{\infty}\left\{\prod_{k=1}^{n}\f{(k-\sigma-1)(2k-3)}{k(2k-1)}\right\}s^{2n}.
\end{eqnarray*}
They can aslo be expressed using Gaussian hypergeometric function by (\ref{D1p}).

\noindent\textbf{Case 2}: $\Delta_1=0$

In this case, (\ref{odeDouglasab}) can be reduced as
\begin{eqnarray}\label{D1za}
\p''(s)=2\sigma\{\ps-s\p'(s)\},\quad\sigma=\pm1,
\end{eqnarray}
The solutions of (\ref{D1za}) are given by
\begin{eqnarray*}
\phi_{\sigma}(s)&=&1+\epsilon
s+2\sum_{n=0}^{\infty}\frac{(-1)^n\sigma^{n+1} s^{2n+2}}{(2n+2)(2n+1)n!}.
\end{eqnarray*}
They can aslo be expressed using error function by (\ref{D1z}).

\noindent\textbf{Case 3}: $\Delta_1<0$

In this case, (\ref{odeDouglasab}) can be reduced as
\begin{eqnarray}\label{D1na}
\left\{1+2\sigma s^2+s^4\right\}\p''(s)=s^2\{\ps-s\p'(s)\},\quad|\sigma|<1.
\end{eqnarray}
The solutions of (\ref{D1na}) are given by (\ref{D1n}). Actually, by Maple we can see that they can aslo be expressed using Gaussian hypergeometric function, but the expression is too complicated. $\phi_0$ is a simper one, it is given by
$$\phi_0=(1+s^4)^{\frac{3}{4}}\mathrm{hypergeom}\left(\left[\f{1}{2},1\right],\left[\f{3}{4}\right],-s^4\right)+\varepsilon s.$$

It is easy to verify that $(p,q)_\p$ are different for any two solutions for the equations (\ref{D1pa}), (\ref{D1za}) and (\ref{D1na}) with different values of $\sigma$ and $\epsilon$. Hence, except Randers type metric and the trivial case, (\ref{D1pa}), (\ref{D1za}) and (\ref{D1na}) provide all the essentially different types of Douglas $\ab$-metrics.

Finally, the required Riemannian metric and $1$-form satisfying (\ref{bijDouglasab}) can be obtained due to Proposition \ref{wienigehiimeig}. In Theorem \ref{douglasab}, we use the special $\b$-deformations given in Theorem \ref{main3}.

\section{Riemannian metrics with conformal $1$-forms}\label{1f}
In order to construct analytical Douglas $\ab$-metrics, we need to construct Riemannian metrics with non-trivial closed and conformal $1$-forms. We don't know whether any given Riemannian metric has such $1$-forms or not. But if it does have, then such property can be  preserved under $\b$-deformations! Specifically, by Proposition \ref{relationunderdeformations} we have the following result.
\begin{lemma}
If $\b$ is closed and  conformal with respect to $\a$, i.e., $\bij=c(x)\aij$, then after $\b$-deformations,
\begin{eqnarray*}
\brij=c\f{\nu(1+2b^2\rho')}{1-b^2\kappa}(\aij-\kappa\bi\bj)+c\left\{\f{\nu(\kappa+b^2\kappa')}{1-b^2\kappa}-4\nu\rho'+2\nu'\right\}\bi\bj.
\end{eqnarray*}
\end{lemma}

Solving the equation
$$\f{\nu(\kappa+b^2\kappa')}{1-b^2\kappa}-4\nu\rho'+2\nu'=0,$$
we immediately obtain the following result.
\begin{proposition}\label{weingiengin}
If $\b$ is closed and conformal but not parallel with respect to $\a$, then $\bb$ is closed and conformal with respect to $\ba$ after $\b$-deforamtions if and only if the deformation factors satisfies
$$\nu=C\sqrt{1-b^2\kappa}e^{2\rho},$$
where $C$ is a non-zero constant.
\end{proposition}

The above fact shows that if $\a$ has closed and conformal $1$-forms, then we can obtain many new metrics having closed and conformal $1$-forms too. In particular, it is known that if $\a$ is of constant sectional curvature, then all its closed and conformal $1$-forms can be completely determined locally\cite{yct-dhfp}. Suppose $h$ is of constant sectional curvature $\mu$, then it can be expressed as
\begin{eqnarray}\label{h}
h=\f{\sqrt{(1+\mu|x|^2)|y|^2-\mu\langle x,y\rangle^2}}{1+\mu|x|^2}
\end{eqnarray}
in some local coordinate system, and its closed and conformal $1$-forms are given by
\begin{eqnarray}\label{Wflat}
W^\flat=\f{\lambda\langle x,y\rangle+(1+\mu|x|^2)\langle
a,y\rangle-\mu\langle a,x\rangle\langle
x,y\rangle}{(1+\mu|x|^2)^\f{3}{2}}
\end{eqnarray}
in the same coordinate system, where $\lambda$ is a constant and $a\in\RR^n$ is a constant vector. In this case, the dual conformal vectors field of $W^\flat$ are very simple and given by
\begin{eqnarray*}
W=\sqrt{1+\mu|x|^2}(\lambda x+a).
\end{eqnarray*}

In particular, take $\mu=0$, $\lambda=1$ and $a=0$, then $h=|y|$ is the standard Euclidean metric and $W^\flat=\xy$. So by Proposition \ref{weingiengin} we have the result below.
\begin{corollary}\label{womeingengiengi}
The Riemmannian metric
$$\ba=e^\rho\sqrt{|y|^2-\kappa\xy^2},$$
has closed and conformal $1$-form expressed as
$$\bb=C\sqrt{1-\kappa\xx}e^{2\rho}\xy,$$
where $\kappa=\kappa(\xx)$ and $\rho=\rho(\xx)$ are two arbitrary functions with an additional condition $1-\kappa\xx>0$. Moreover, $\bar b^2=C^2e^{2\rho}\xx$.
\end{corollary}
Notice that parts of Riemannian metrics in Corollary \ref{womeingengiengi} have constant sectional curvature. For instant, if $\kappa=\frac{\mu}{1+\mu\xx}$ and $e^{2\rho}=\f{1}{1+\mu\xx}$, then $\ba$ is just the metric (\ref{h}) and $\bb$ is just the $1$-form (\ref{Wflat}) with $\lambda=C$ and $a=0$. If $\kappa=\frac{\mu}{1+\mu\xx}$ and $e^{2\rho}=1$, then
$$\ba=\f{\sqrt{(1+\mu\xx)\yy-\mu\xy^2}}{\sqrt{1+\mu\xx}}$$
also has constant sectional curvature $-\mu$. Even so, most of them have non constant sectional curvature.

%\begin{figure}[htbp]
 % \includegraphics[width=0.80\textwidth]{a.pdf}
 % \caption{}\label{tu}
%\end{figure}

\noindent Changtao Yu\\
School of Mathematical Sciences, South China Normal
University, Guangzhou, 510631, P.R. China\\
aizhenli@gmail.com


\begin{thebibliography}{00}
\bibitem{aim-ttos}
P. L. Antonelli, R. S. Ingarden and M. Mastumoto, {\it The theory of sprays and Finsler spaces with application in physics and biology}, Kluwer Academic Publishers, 1993.
\bibitem{bacs-cxy-szm-curv}
S. B\'{a}cs\'{o}, X. Cheng and Z. Shen, {\it Curvature properties of
$\ab$-metrics}, In ``Finsler Geometry, Sapporo 2005-In Memory of
Makoto Matsumoto", ed. S. Sabau and H. Shimada, Advanced Studies in
Pure Mathematics {\bf 48}, Mathematical Society of Japan, 2007, 73-110.

\bibitem{BM}
S. B\'{a}cs\'{o} and M. Matsumoto, {\it On Finsler spaces of Douglas type. A
generalization of the notion of Berwald space}, Publ. Math. Debrecen,
{\bf 51} (1997), 385-406.

\bibitem{db-robl-szm-zerm}
D. Bao, C. Robles and Z. Shen, {\it Zermelo navigation on Riemannian manifolds}, J. Diff. Geom., {\bf 66} (2004), 391-449.

\bibitem{Be}
L. Berwald, {\it \"{U}ber die n-dimensionalen Geometrien konstanter Kr\"{u}mmung, in denen die Geradendie k\"{u}rzesten sind}, Math. Z., {\bf 30} (1929), 449-469.

%\bibitem{CS2}
%S. S. Chern and Z. Shen, {\it Riemann-Finsler
%Geometry}, World Scientific, 2005.

%\bibitem{CS}
%X. Chen and Z. Shen, {\it On Douglas Metrics}, Publ. Math. Debrecen, {\bf 66} (2005), 503-512.

\bibitem{cl}
Z.Chang and X. Li, {\it Modified Newton's gravity in Finsler space as a possible alternative to dark matter hypothesis}, Phy. Lett. B, {\bf 668} (2008), 453-456.

\bibitem{CT}
X. Cheng and Y. Tian, {\it Ricci-flat Douglas $\ab$-metrics}, Publ. Math Debrecen, {\bf 30} (2012), 20-32.

\bibitem{D}
J. Douglas, {\it The general geometry of paths}, Ann. of Math, {\bf 29} (1927-28), 143-168.

\bibitem{Li}
B. Li, Z. Shen and Y. Shen, {\it On a class of Douglas metrics}, Studia Sci. Math. Hungarica, {\bf 46} (2009), 355-365.

\bibitem{Mo}
X. Mo, N. M. Sol\'{o}rzano and K. Tenenblat, {\it On spherically symmetric Finsler metrics with vanishing Douglas
curvature}, Diff. Geom. Appl. {\bf 31} (2013), 746-758.

\bibitem{Ran}
G. Randers, {\it On an asymmetric in the four-space of general relativily}, Phys. Rev., {\bf 59} (1941), 195-199.

\bibitem{SSZ}
E. Sevim, Z. Shen and L. Zhao, {\it On a class of Ricci-flat Douglas metrics}, Int. J. Math., {\bf 23} (2012), 1250046, 15 pp.

\bibitem{szm-yct-oesm}
Z. Shen and C. Yu, {\it On Einstein square metrics}, Publ. Math. Debrecen, {\bf 85}(3-4) (2014), 413-424.

\bibitem{YGJ}
G. Yang, {\it On a class of two-dimensional Douglas and projectively flat Finsler metrics}, The Scientific World Journal, {\bf 2013} (2013), 291491, 11 pp.

\bibitem{yct-dhfp}
C. Yu, {\it Deformations and Hilbert's Fourth Problem}, Math. Ann., {365} (2016) 1379-1408

\bibitem{yct-odfa}
C. Yu, {On dually flat $\ab$-metrics}, J. Math. Anal. Appl. {412} (2014) 664-675.

\bibitem{yct-odfr}
C. Yu, {\it On dually flat Randers metrics},  Nonlinear Anal. {95} (2014) 146-155.

\bibitem{YCT}
C. Yu and H. Zhu, {\it On a new class of Finsler metrics}, Diff. Geom. Appl., {\bf 29} (2011), 244-254.
\end{thebibliography}
\end{document}